\documentclass[reqno,11pt]{amsart}

\usepackage{amssymb,amscd,amsthm,latexsym, amsmath}
\usepackage[all]{xy}

\theoremstyle{plain}
\newtheorem{Lemma}{Lemma}[section]
\newtheorem{theorem}[Lemma]{Theorem}
\newtheorem{lemma}[Lemma]{Lemma}
\newtheorem{proposition}[Lemma]{Proposition}
\newtheorem{corollary}[Lemma]{Corollary}
\theoremstyle{definition}
\newtheorem{definition}[Lemma]{Definition}
\newtheorem{example}[Lemma]{Example}
\newtheorem{examples}[Lemma]{Examples}

\theoremstyle{remark}

\DeclareMathOperator{\End}{End}

\newdir{ >}{{}*!/-8pt/@{>}}
\newcommand{\Cal}[1]{{\mathcal #1}}

\def\cleft{\hbox{[\kern-.16em\hbox{[}}}
\def\cright{\hbox{]\kern-.16em\hbox{]}}}

\newcommand{\Gp}{\mathsf{Gp}}
\newcommand{\Set}{\mathsf{Set}}

\DeclareMathOperator{\Aut}{Aut}

\renewcommand{\phi}{\varphi}
\DeclareMathOperator{\op}{\rm op}

\DeclareMathOperator{\id}{id}
\numberwithin{equation}{section}

\makeatletter
\@namedef{subjclassname@2020}{%
	\textup{2020} Mathematics Subject Classification}

\makeatother

\begin{document}

\title{Trusses, ditrusses, weak trusses}

\author{Alberto Facchini}
\address{Dipartimento di Matematica ``Tullio Levi-Civita'', Universit\`a di Padova, 35121 Padova, Italy}
\email{facchini@math.unipd.it}


 \begin{abstract}   In this paper we extend to left skew trusses $(T,+,\circ,\sigma)$ previous work on left skew rings. We had presented a left skew ring as a group $(N,+)$ with two binary operations $\circ$ and $\cdot$ with $\circ$ associative, $\cdot$ left distributive over the addition $+$ of the group, and such that the difference of the two operations $\circ$ and $\cdot$ is the binary operation $\pi_1\colon N\times N\to N$. Here we extend this idea to the left skew trusses introduced in 2019 by Brzezi\'nski, replacing the operation $\pi_1$ with the binary operation $\sigma\pi_1\colon T\times T\to T$. The case where  the semigroup morphism $\lambda^T\colon T\to \End_\Gp(T,+)$ is constant turns out to be particular interesting. We get several canonical category isomorphisms. For instance, we get a category isomorphism between  the category of all left skew trusses $(T,+,\circ,\sigma)$ with $\lambda^T\colon (T,\circ)\to \End_\Gp(T,+)$ a constant semigroup morphism and $\sigma,\lambda^T_0$ image-commuting idempotent endomorphisms
and the category of all associative interchange near-rings. Interchange near-rings were introduced by Edmunds in 2016. When $\sigma$ is an idempotent group endomorphism of the group $(T,+)$ and $\lambda^T\colon (T,\circ)\to \End_\Gp(T,+)$ is a semigroup morphism constantly equal to a group endomorphism $\tau$, we also get a sort of duality exchanging the mappings $\sigma$ and $\tau$.
\end{abstract}

	\maketitle
	
	\section{Introduction}
	
	Left skew braces were introduced by Guarnieri and Vendramin \cite{GV} in the context of the study of set-theoretic solutions of the Yang-Baxter equation.  Left skew braces---or, more precisely, left skew rings \cite{Fac, RumpJAA2019}---can be viewed 
	as modified left near-rings. Specifically, a left skew ring is a triple $(G,+,\circ)$ where $(G,+)$ is a (not necessarily abelian) group, $(G,\circ)$ is a semigroup, and the {\em left skew distributive law} holds, that is, $a\circ (b + c) = (a\circ b)-a+ ( a\circ c)$ for every $a,b,c\in G$. 
In the case of left skew braces, we additionally require that $(G,\circ)$ is a group.

Left skew rings can be also equivalently described as quadruples $(G,+,\circ,\cdot)$, where $(G,+)$ is a group, $\circ $ and $\cdot $ are binary operations, $\circ$ is associative, $\cdot$~is left distributive over $+$, and the difference between $\circ$ and $\cdot$ in the additive group of all binary operations on $G$ equals the operation $\pi_1\colon G\times G\to G$, the projection 
map onto the first factor \cite{Fac}.

The notion of left skew ring was generalized by Brzezi\'nski  to that of a  left skew truss in~\cite{Bre}. A {\em left skew truss} is a quadruple $(T, +,\circ,\sigma)$, where $(T, +) $ is a group, $(T,\circ)$ is a semigroup, $\sigma\colon T\to T$ is a unary operation, and the {\em left skew $\sigma$-distributive law} holds, that is $a\circ (b + c) = (a\circ b)-\sigma(a)+ ( a\circ c)$ for every $a,b,c\in T$.

In this paper, we extend the framework developed for left skew rings in \cite{Fac} to the setting of left skew trusses. We show that any left skew truss can be viewed as a quintuple $(G,+,\circ,\cdot,\sigma)$ where $(G,+)$ is a group, $\circ$ and $\cdot$ are binary operations, $\sigma$ is a unary operation, $\circ$ is associative, $\cdot$~is distributive over $+$,
and the difference $\circ - \cdot$ is the operation $\sigma\pi_1\colon G\times G\to G$ (Subsection~\ref{subs}). 
A particularly nice situation arises when  the mapping $\sigma\colon G\to G$ is an idempotent group endomorphism of the additive group $(G,+)$ (Subsection~\ref{First}). 

For any left skew truss $(T, +,\circ,\sigma)$ and each $a\in T$, let 
$\lambda^T_a\colon T\to T$ be the mapping defined by $
\lambda^T_a(b)=-\sigma(a)+(a\circ b)\ \mbox{\rm for every }b\in T.$ Then: $\lambda^T_a$ is a group endomorphism of $(T,+)$; we have $a\circ 0_T = \sigma(a)$ for all $a\in T$; and $\sigma$ is an idempotent mapping. Moreover, if $\sigma(0_T) = 0_T$, then $\lambda^T_{0_T}$ is necessarily idempotent, and $0_T\circ a = \lambda^T_0(a)$ for all $a\in T$ (Theorems~\ref{2.3} and~\ref{2.3'}).
We define a {\em left ditruss} as a quintuple $(D,+,\sigma,\circ,\cdot)$ where $+$, $\circ$, and $\cdot$ are three binary operations, $\sigma$ is unary, $(D,+)$ is a group, and $\sigma(a)+a\cdot b=a\circ b$ for every $a,b\in D$.
Given a left skew truss $(T, +,\circ,\sigma)$, defining the operation $\cdot $ setting $a\cdot b=-\sigma(a)+a\circ b$ for all $a,b\in T$, yields a left ditruss $(T, +,\sigma, \circ,\cdot)$.

Under mild conditions on a ditruss $(D,+,\sigma,\circ,\cdot)$, we show (Theorems \ref{new} and \ref{2.3'}, and Proposition~\ref{2.3''}):

\noindent (1)  the operation $\circ$ is associative if and only if the operation $\cdot $ is {\em left weakly $\sigma$-associative},  i.e., $$(\sigma(a)+a\cdot b)\cdot c=a\cdot(b\cdot c)\ \mbox{ for all }a,b,c\in D;$$

\noindent (2)  the operation $\cdot$ is distributive if and only if the operation $\circ$ is {\em left skew $\sigma$-distributive}, that is, $a\circ(b+c)=a\circ b-\sigma(a)+a\circ c$ for all $a,b,c\in D$; and 

\noindent (3) $\sigma$ and $\lambda^D_{0_D} $ are commuting idempotent group endomorphisms of the group $(D,+)$. 

\noindent Moreover, left weak $\sigma$-associativity of $\cdot$ is equivalent to requiring that the map
$$\lambda^D\colon D\to \End_\Gp(D,+)$$ is a semigroup homomorphism from $(D,\circ)$ to the semigroup $\End_\Gp(D,+)$ (under composition).

Among semigroup morphisms $\varphi\colon D \to E$ between two semigroups $D\ne\emptyset$ and $E$, 
the most natural are the constant morphisms, i.e.,  the morphisms $\varphi\colon D\to E$ 
for which $\varphi(a)=\varphi(b)$ for all $a,b\in D$. These constant morphisms $D\to E$ are in one-to-one correspondence with the idempotent elements of $E$. 
Studying the case where  the semigroup morphism $\lambda^D\colon D\to \End_\Gp(D,+)$ is constant leads us to a connection with associative interchange near-rings (Section~\ref{Inter}). An  
 {\em interchange near-ring} is a triple $(G,+,\circ)$ where $(G,+)$ is a group and $\circ$ is a binary operation on $G$ satisfying the {\em interchange law}:
$$(w+x)\circ(y+z)=(w\circ y)+(x\circ z)\ \mbox{\rm for every }w,x,y,z\in G.$$ Interchange near-rings were introduced in \cite{Australian}.
Two endomorphisms $\varepsilon, \eta$ of an additive group $G$ are said to be {\em image-commuting} if $\varepsilon(x)+\eta(y)=\eta(y)+\varepsilon(x)$ for all $x,y\in G$. 
By \cite[Theorem~3.5]{Australian}, the interchange near-ring structures on any group $(G,+)$ are all of the form $(G,+, \varepsilon\pi_1+\eta\pi_2)$ for a pair $(\varepsilon,\eta)$ of image-commuting group endomorphisms of 	$(G,+)$. 
We find natural category isomorphisms. The category of all left ditrusses $(D,+,\sigma, \circ,\cdot)$ with $\circ$ associative, constant semigroup morphism $\lambda^D\colon (D,\circ)\to \End_\Gp(D,+)$, and $\sigma$, $\lambda^D_{0}$ image-commuting idempotent endomorphisms of $(D,+)$, is isomorphic to the category of associative interchange near-rings (Theorem~\ref{bjip}).
Furthermore, the functor $$(D,+,\sigma, \circ,\tau\pi_2)\mapsto (D,+,\tau, \circ,\sigma\pi_2)$$ is an involutive categorical automorphism of the category of left ditrusses $(D,+,\sigma,\circ,\cdot)$ with $\circ$ associative and $\lambda^D\colon (D,\circ)\to \End_\Gp(D,+)$ a constant semigroup morphism. 
We define a {\em left weak truss} as a quadruple $(W,+,\cdot,\sigma)$ where:
$(W,+)$ is a group,
$\sigma$ is a unary operation,
$\cdot$ is a binary operation that is left weakly $\sigma$-associative,
and 
$\cdot$ is left distributive:
$$(\sigma(a)+a\cdot b)\cdot c=a\cdot(b\cdot c)\ {\mbox{\rm and }}a(b+c)=ab+ac\ {\mbox{\rm for every }}a,b,c\in W.$$
The category of left skew trusses $(T,+,\circ,\sigma)$ with $\sigma$ an idempotent group endomorphism of $(T,+)$ is canonically isomorphic to the category of left weak trusses $(W,+,\cdot,\sigma)$ with $\sigma$ an idempotent group endomorphism of $(W,+)$ (Theorem~\ref{main}). This isomorphism associates to each left skew truss $(T, +,\circ,\sigma)$ with $\sigma$ an idempotent group endomorphism the left weak truss $(T, +,\cdot,\sigma)$, where $\cdot$ is the binary operation on $T$ defined by $a\cdot b=-\sigma(a)+a\circ b$ for all $a,b\in T$.

\medskip

Throughout this paper, by ``algebra'' we mean ``algebra'' in the sense of Universal Algebra, and we sometimes denote algebras in an informal way. For instance, an additive group will be denoted both in the form $(G,+,-,0_G)$ as one usually correctly does in Universal Algebra and in the more common form $(G,+)$. Similarly, a near-ring will be denoted in both forms $(R,+,-,0_R,\cdot)$ and $(R,+,\cdot)$. Also, when we say that a binary operation $\cdot$ on an additive group $(G,+)$ is ``left distributive'', we mean that $\cdot$ distributes over addition of the group $G$, i.e., that $a\cdot(b+c)=a\cdot b+a\cdot c$ for every $a,b,c\in G$.

		\section{Basic terminology on trusses and idempotent endomappings}\label{4.3}  
	
\subsection{Left skew trusses}\label{SS2.1} The definition of skew trusses is due to Tomasz Brzezi\'nski \cite[Section~2]{Bre}.	A {\sl left skew truss} is
	an algebra $(T, +,-,0_T,\circ,\sigma)$, where $(T, +,-,0_T) $ is a group, $\circ$ is a binary operation on $T$, $\sigma$ is a unary operation, and the following two axioms are satisfied:\medskip 
		
		\noindent (a) ({\em associativity of the operation $\circ$})\qquad\qquad $a\circ (b \circ c) = (a\circ b)\circ c,$ 
		
		\medskip 
		
		and
		
		\medskip 

\noindent (b) ({\em left skew $\sigma$-distributivity})\qquad\qquad $a\circ (b + c) = (a\circ b)-\sigma(a)+ ( a\circ c)$ \qquad\qquad\qquad\qquad(3.1)
		
		\medskip 
		
		\noindent 
		for every $a,b,c\in T$.
		
\bigskip

The first examples of left skew trusses are:

\noindent (1) For the unary operation $\sigma\colon T\to T$ defined by $\sigma(a)=0_T$ for all $a\in T$, left skew trusses $(T, +,-,0_T,\circ,\sigma)$ are exactly left near-rings.

\noindent (2) For $\sigma=\id_T$ (the identity mapping $\id_T\colon T\to T$), left skew trusses $(T, +,-,0_T,\circ,\id_T)$ are exactly the left skew rings studied in \cite{Fac}.

\subsection{Idempotent endomappings}\label{subs} Let $T$ be a set. In \cite[Subsection 4.1]{Fac} we saw that a particularly important binary operation on $T$ is the operation $\pi_1\colon T\times T\to T$ defined by $a\,\pi_1\, b:=a$ for every $a,b\in T$. In fact, the set of all binary operations on $T$ turns out to be a monoid with respect to a suitably defined multiplication of binary operations, and $\pi_1$ turns out to be the identity of this monoid  \cite[Theorem~1.2]{LPOMR}. For any unary operation $\sigma\colon T\to T$ on the set $T$, the binary operation on $T$ naturally corresponding to $\pi_1$  is the operation $\sigma\pi_1\colon T\times T\to T$ defined by $a(\sigma\pi_1)b:=\sigma(a)$ for every $a,b\in T$. It is now important to notice that:

\begin{lemma}\label{1} Let $\sigma\colon T\to T$ be a unary operation on a set $T$. The binary operation $\sigma\pi_1$ on $T$ is associative if and only if the mapping $\sigma$ is idempotent.\end{lemma}

\begin{proof} For every $a,b,c\in T$, we have $$a\,(\sigma\pi_1)(b\,(\sigma\pi_1)\,c)=\sigma(a)\quad{\mbox{\rm and}}\quad (a\,(\sigma\pi_1)\,b)(\sigma\pi_1)\,c=\sigma(a\,(\sigma\pi_1)\,b)=\sigma(\sigma(a)).$$\end{proof} 

\bigskip

Since we have met with idempotent endomappings $\sigma$ on $T$, that is, mappings $\sigma$ of a set $T$ into itself such that $\sigma^2=\sigma$, let me briefly recall here the notion of semidirect product in Universal Algebra \cite{semi}. 

\bigskip

We can construct an  {\em  internal semidirect-product decomposition} of any algebra $A$ making use of any subalgebra $B$ of $A$ and any congruence $\omega$ on $A$ such that the intersection of $B$ with any congruence class of $A$ modulo $\omega$ is a singleton \cite{semi}. Internal semidirect-product decompositions $A=B \ltimes\omega$ of an algebra $A$ are in one-to-one correspondence with the idempotent endomorphisms of the algebra $A$. If $\sigma\colon A\to A$ is an idempotent endomorphism of an algebra $A$, the corresponding internal semidirect-product decomposition of $A$ is $A=\sigma(A) \ltimes\!\!\sim_\sigma$, where $\sim_\sigma$ is the kernel of the endomorphism $\sigma$, that is, the congruence $\sim_\sigma$ on $A$ defined, for every $a,a'\in A$, by $a\sim_{\sigma}a'$ if $\sigma(a)=\sigma(a')$. For instance, in the case of sets (algebras with no operations), we have that, for a set $A$, idempotent endomappings of $A$ are in one-to-one correspondence with semidirect-product decompositions $A=B \ltimes\!\!\sim$ of $A$, that is, with the set of all pairs $(B,\sim)$, where $\sim$ is any equivalence relation on $A$, and $B$ is an irredundant set of representatives of $A$ modulo $\sim$. Given any such pair $(B,\sim)$, the corresponding idempotent endomapping $\sigma\colon A\to A$ maps any element $a$ of $A$ to the unique element of $B\cap[a]_\sim$. Conversely, conversely, given any idempotent endomapping $\sigma\colon A\to A$, the corresponding pair is $(\sigma(A),\sim_\sigma)$.

\medskip

{\em External semidirect products} (or, better, {\em semidirect-product extensions} of $B$) are constructed from, and parametrized by, all product-preserving functors $F\colon \Cal C_B\to \Set_*$ from a suitable category $\Cal C_B$ containing $B$ (called the {\em enveloping category} of $B$ or the {\em term category }of $B$) to the category $\Set_*$ of pointed sets \cite{semi}. The objects of $\Cal C_B$ are the $n$-tuples $(b_1,\dots,b_n)$ of elements of $B$, with $n\ge 0$. (That is, the objects of $\Cal C_B$ are all words in the alphabet $B$.) A morphism $(b_1,\dots,b_n)\to (c_1,\dots,c_m)$ in $\Cal C_B$ is any $m$-tuple $(p_1,\dots,p_m)$ of $n$-ary terms \cite[Definitions~10.1 and~10.2]{BS} such that $p_j(b_1,\dots,b_n)=c_j$ for every $j=1,2,\dots,m$. Composition of morphisms is defined by 
$$(q_1,\dots,q_r)\circ(p_1,\dots,p_m)=(q_1(p_1,\dots,p_m), q_2(p_1,\dots,p_m),\dots, q_r(p_1,\dots,p_m)).$$ See \cite[Section~4]{semi}. Two objects $(a_1,\dots,a_n),(b_1,\dots,b_m)$ are isomorphic in $\Cal C_B$ if and only if there is a bijection $\varphi\colon \{1,2,\dots,n\}\to\{1,2,\dots,m\}$ such that $a_{\varphi(j)}=b_j$ for every $j=1,\dots,n$.

In the category $\Cal C_B$ the product of two objects $(a_1,\dots,a_n),(b_1,\dots,b_m)$ of $\Cal C_B$ is the $(n+m)$-tuple $(a_1,\dots,a_n,b_1,\dots,b_m)$ (justapposition of words), with projections $$\pi_a\colon (a_1,\dots,a_n,b_1,\dots,b_m)\to (a_1,\dots,a_n)\quad{\mbox{\rm and}}\quad\pi_b\colon (a_1,\dots,a_n,b_1,\dots,b_m)\to (b_1,\dots,b_m),$$ the tuples of terms $\pi_a=(x_1,\dots,x_n)$ and $\pi_b=(x_{n+1},\dots,x_{n+m})$. In order to check that the Universal Property of Product holds, fix any object $(c_1,\dots,c_r)$ of $\Cal C_B$ and any pair of morphisms $f_a\colon (c_1,\dots,c_r)\to (a_1,\dots,a_n)$ and $f_b\colon (c_1,\dots,c_r)\to(b_1,\dots,b_m)$, where $f_a=(p_1,\dots,p_n)$ and $f_b=(q_1,\dots,q_m)$ are tuples of $r$-ary terms such that $p_j(c_1,\dots,c_r)=a_j$ for every $j=1,\dots,n$ and $q_k(c_1,\dots,c_r)=b_k$ for every $k=1,\dots,m$. It is easily seen that the unique morphism $f\colon (c_1,\dots,c_r)\to  (a_1,\dots,a_n,b_1,\dots,b_m)$ such that $\pi_af=f_a$ and $\pi_bf=f_b$ is the $(n+m)$-tuple of terms $f=(p_1,\dots,p_n,q_1,\dots,q_m)$.

Also, $\Cal C_B$ has a terminal object (the $0$-tuple, i.e., the word of length $0$). Thus the category $\Cal C_B$ has finite products. 

\begin{example} {\em (The variety of sets)} \ \ In the case of a set (an algebra with no operations), we have that, for a set $A$, the objects of the category $\Cal C_A$ are all $n$-tuples $(a_1,\dots,a_n)$ of elements of $A$, for any integer $n\ge0$. The terms in the variables $x_1,\dots,x_n$ are only $x_1,\dots,x_n$, so that a morphism $(a_1,\dots,a_n)\to (b_1,\dots,b_m)$ in $\Cal C_A$ is any $m$-tuple $(x_{\varphi(1)},\dots,x_{\varphi(m)})$ with $\varphi\colon \{1,2,\dots,m\}\to\{1,2,\dots,n\}$ a mapping such that $(a_{\varphi(1)},\dots, a_{\varphi(m)})=(b_1,\dots,b_m)$. The composite morphism of the morphisms $$(x_{\varphi(1)},\dots,x_{\varphi(m)})\colon (a_1,\dots,a_n)\to (b_1,\dots,b_m)$$ and $$(x_{\psi(1)},\dots,x_{\psi(r)})\colon (b_1,\dots,b_m)\to (c_1,\dots,c_r)$$  is $$(x_{\psi(1)},\dots,x_{\psi(r)})\circ(x_{\varphi(1)},\dots,x_{\varphi(m)})=(x_{\varphi\psi(1)},\dots,x_{\varphi\psi(r)}).$$ Here $\varphi\colon \{1,2,\dots,m\}\to\{1,2,\dots,n\}$ and $\psi\colon \{1,2,\dots,r\}\to\{1,2,\dots,m\}$ are mappings such that  $(a_{\varphi(1)},\dots, a_{\varphi(m)})=(b_1,\dots,b_m)$ and $(b_{\psi(1)},\dots, b_{\psi(r)})=(c_1,\dots,c_r)$.

In the case of $1$-tuples, that is, of elements of $A$, we have that in the category $\Cal C_A$ there is no morphism $a\to a'$ for $a\ne a'$ ($a,a'\in A$), and there is only the identity morphism $a\to a'$ for $a=a'$. Therefore, in order to describe a product-preserving functor $F\colon \Cal C_A\to\Set_*$ it suffices to associate to any $a\in A$ a pointed set $(B_a,b_a)$. Then to any $n$-tuple $(a_1,\dots,a_n)$ there corresponds via $F$ the pointed set $(B_{a_1}\times\dots\times B_{a_n},(b_{a_1},\dots,b_{a_n}))$, and to any morphism $$(x_{\varphi(1)},\dots,x_{\varphi(m)})\colon (a_1,\dots,a_n)\to (a'_1,\dots,a'_m)$$ there corresponds the mapping $(\pi_{\varphi(1)}\times\dots\times\pi_{\varphi(m)})\colon B_{a_1}\times\dots\times B_{a_n}\to B_{a'_1}\times\dots\times B_{a'_m}$. The semidirect-product extension corresponding to this functor $F\colon \Cal C_A\to\Set_*$ is the disjoint union $\dot{\bigcup}_{a\in A}B_a$. Thus product-preserving functors $F\colon \Cal C_A\to\Set_*$ parametrize all semidirect-product extensions of the set $A$. Notice that the idempotent endomapping corresponding to $\dot{\bigcup}_{a\in A}B_a$ is the mapping $e\colon \dot{\bigcup}_{a\in A}B_a\to \dot{\bigcup}_{a\in A}B_a$ that maps all the elements of $B_a$ to $b_a$. Hence $\dot{\bigcup}_{a\in A}B_a=A\ltimes\omega$, where $\omega$ is the equivalence relation on $\dot{\bigcup}_{a\in A}B_a$ with equivalence classes the sets $B_a$.
\end{example}

\subsection{First properties of left skew trusses}\label{First}

\begin{lemma}\label{lemma} Let $(T,+)$ be an additive group and $\sigma\colon T\to T$ be a mapping. Then the binary operation $\sigma\pi_1$ is:

\noindent{\rm (a)} Right distributive if and only if $\sigma$ is a group morphism.

\noindent{\rm (b)}  Always left skew $\sigma$-distributive.

\noindent{\rm (c)} Left distributive if and only if $\sigma(a)=0_T$ for every $a\in T$. \end{lemma}

The proof is elementary.

\smallskip

In view of Lemmas \ref{1} and \ref{lemma}(b), if $(T,+)$ is any group and $\sigma\colon T\to T$ is any mapping, then $(T,+,\sigma\pi_1,\sigma)$ is a left skew truss if and only if the mapping $\sigma\colon T\to T$ is idempotent. More generally:

		\begin{theorem}\label{2.3} If $(T, +,-,0_T,\circ,\sigma)$ is any left skew truss, then:
		
		\noindent{\rm (a)} For every $a\in T$, the mapping $\lambda^T_a\colon T\to T$, defined by
$\lambda^T_a(b)=-\sigma(a)+(a\circ b)$ for every $b\in T$, is a group endomorphism of the group $(T,+)$.
		
		\noindent{\rm (b)} $a\circ 0_T=\sigma(a)$ for every $a\in T$.

		\noindent{\rm (c)} $\sigma\colon T\to T$ is an idempotent mapping.
		
		\noindent{\rm (d)}  If $\sigma(0_T)=0_T$, then the group endomorphism $\lambda^T_{0_T}$ of $(T,+)$ is idempotent, $0_T\circ a=\lambda^T_0(a)$ for every $a\in T$, and the idempotent mappings $\sigma$ and $\lambda_{0_T}^T$ commute. 		\end{theorem}

		\begin{proof} (a)  is equivalent to $\lambda^T_a(b+c)=\lambda^T_a(b)+\lambda^T_a(c)$ for every $a,b,c\in T$, and this is trivially equivalent to  Identity (3.1). 
		
		(b) follows from (a), because $\lambda_a(0_T)=0_T$, so that $0_T=-\sigma(a)+(a\circ 0_T)$.
		
		(c) follows from (b) and the associativity of $\circ$.

		As far as (d) is concerned, suppose that $\sigma(0_T)=0_T$. Then $\lambda^T_{0_T}(a)=-\sigma(0_T)+(0_T\circ a)=0_T\circ a$ for every $a\in T$. Then $\lambda^T_{0_T}(\lambda^T_{0_T}(a))=0_T\circ 0_T\circ a=\sigma(0_T)\circ a=0_T\circ a=\lambda^T_{0_T}(a)$ for every $a\in T$. Finally, $\sigma$ and $\lambda_{0_T}^T$ commute because $0_T\circ(a\circ 0_T)=(0_T\circ a)\circ0_T$.
		\end{proof}

		As always in Universal Algebra, left skew truss homomorphisms are defined to be the mappings that preserve all operations $+,-,0_T,\circ,\sigma$, but in view of Theorem \ref{2.3}, it suffices that they preserve addition $+$ and multiplication $\circ$. Left skew trusses form a variety in
the sense of Universal Algebra. It is a variety of $\Omega$-groups in the sense of Higgins, it is a semiabelian category. 

\medskip

For every left skew truss $(T,+,\circ,\sigma)$ with $\sigma(0_T)=0_T$,  the group endomorphism $\lambda^T_{0_T}$ of $(T,+)$ is idempotent (Theorem~\ref{2.3}(d)), so that the additive group $(T,+)$ has a semidirect-sum decomposition $T = T_0 \rtimes T_c$, where $T_0:=\ker(\lambda_{0_T}^T)=\{\,a\in T \mid 0_T\circ a = 0_T\,\}$ (the  {\em $0$-symmetric part} of $T$) is a normal subgroup of  $(T,+)$, and $T_c:=\lambda^T_{0_T}(T) = \{\,a\in T \mid 0_T\circ a = a\,\}$ (the {\em constant part} of $T$) is a subgroup of the additive group $T$. By making use of Theorem~\ref{2.3}, it is easy to check that both $T_0$ and $T_c$ are left skew subtrusses of $(T,+,\circ,\sigma)$, that is, they are also closed for $\circ$ and $\sigma$. 

\medskip

		An {\em ideal} of the left skew truss $(T, +,\circ,\sigma)$ is a normal subgroup $I$ of the additive group $(T, +)$ such that $(i+a)\circ b-a\circ b\in I$ and $\lambda_a(i)\in I$ for every $a,b\in T$ and  $i\in I$. (Notice that, for a normal subgroup $I$ of the additive group $T$, $(a+i)\circ b-a\circ b\in I$ for every $i\in I$ if and only if $(i+a)\circ b-a\circ b\in I$ for every $i\in I$.)
		
		\begin{proposition} Let $(T, +,\circ,\sigma)$ be a left skew truss. There is a one-to-one correspondence between the set of all ideals of $T$ and the set of all congruences of the left skew truss $T$.\end{proposition}
	
	\begin{proof} Clearly, it suffices to prove that, for a normal subgroup $I$ of the additive group of $T$, the equivalence $\equiv$ on $T$, defined, for every $a,b\in T$, by $a\equiv b$ if $a-b\in I$,  is compatible not only with the addition $+$, but also with the multiplication $\circ$ and the unary operation $\sigma$ if and only if  $(i+a)\circ b-a\circ b\in I$ and $\lambda_a(i)\in I$ for every $a,b\in T$ and $i\in I$.
		
Suppose both $\circ$ and $\sigma$ compatible with the equivalence $\equiv$. Then $(i+a)\circ b\equiv (0+a)\circ b=a\circ b$, so that $(i+a)\circ b-a\circ b\in I$. Moreover, for every $a,b\in T$ and every $i\in I$, we have that $a\circ (b+i)\equiv a\circ b$, so that $-(a\circ b)+a\circ (b+i)\in I$. Hence, by left skew $\sigma$-distributivity of $\circ$, we get that $-(a\circ b)+a\circ b-\sigma (a)+a\circ i\in I$, that is $\lambda_a(i)=-\sigma (a)+a\circ i\in I$.

Conversely, assume that $(i+a)\circ b-a\circ b\in I$ and $\lambda_a(i)\in I$ for every $i\in I$, $a,b\in T$. Suppose $a,b,c,d\in T$, $a\equiv b$ and $c\equiv d$. Then $b=i+a$ and $d=c+j$ for suitable $i,j\in I$. Therefore $b\circ d=(i+a)\circ (c+j)=(i+a)\circ c -\sigma(i+a)+(i+a)\circ j\equiv a\circ c+\lambda_{i+a}(j)\equiv a\circ c$, so that $\equiv$ is compatible with $\circ$. Finally, $\equiv$ is compatible with $\sigma$, that is, $a\equiv b$ implies $\sigma(a)\equiv \sigma(b)$, because $\sigma(a)=a\circ 0_T\equiv b\circ 0_T=\sigma(b)$.\end{proof}

\section{Subtraction of operations and left ditrusses}\label{4}

\subsection{Subtraction of operations and left ditrusses}
Extending the framework developed for left skew rings in \cite{Fac} to the setting of left skew trusses, we will now write the binary operation $\sigma\pi_1$ on an additive group $(T,+)$ as the difference of an associative operation and a left distributive operation. That is,  for a group  $(T,+)$, and a mapping $\sigma\colon T\to T$ we will look for the triples of operations $(\sigma,\circ,\cdot)$ on $T$ with $\sigma$ unary, $\circ$ associative, $\cdot$ left distributive, and $\sigma\pi_1=\circ-\cdot$. This is equivalent to $\sigma(a)=a\circ b-a\cdot b$ for every $a,b\in T$, or to $\sigma(a)+a\cdot b=a\circ b$ for every $a,b\in T$.

\begin{definition} A {\em left ditruss} is an algebra $(D,+,-,0,\sigma,\circ,\cdot)$, where  $+$, $\circ$ and $\cdot$ are three binary operations, $-$ and $\sigma$ are unary, and $0$ 
is nullary, satisfying the following conditions:

(1) $(D,+,-,0)$ is a group, not-necessarily abelian; and

(2) \begin{equation}\sigma(a)+a\cdot b=a\circ b\label{skew}\end{equation} for every $a,b\in D$.\end{definition}

As usual in Universal Algebra, left ditruss morphisms are the mappings compatible with all the operations $+,-,0,\sigma,\circ$ and $\cdot$ of ditrusses, but it is clear that a mapping is a left ditruss morphism if and only if it is compatible with $+$ and with any two of the three operations $\sigma,\circ$ and $\cdot$.

\begin{example}\label{cvhli} A particularly interesting full subcategory of the category of left ditrusses is the category $\Cal C$ of the left ditrusses $(D,+,-,0,\sigma,\circ,\cdot)$ for which $\cdot$ depends only on the second factor:

\begin{lemma}\label{bpij} The following conditions are equivalent for a left ditruss $(D,+,-,0,\sigma,\circ,\cdot)$:

{\rm (a)} The binary operation $\cdot$ depends only on the second factor;

{\rm (b)}  $a\cdot b=\tau(b)$ for some mapping $\tau\colon D\to D$;

{\rm (c)}  $\cdot=\tau\pi_2$ for some mapping $\tau\colon D\to D$;

{\rm (d)} The mapping $\lambda\colon D\to D^D$ defined by $\lambda\colon a\mapsto \lambda_a$, where $\lambda_a(b)=a\cdot b$ for every $a,b\in D$, is constant.\end{lemma}

The proof of Lemma \ref{bpij} is elementary. Let $\Cal C$ be the category of all left ditrusses satisfying the equivalent conditions of Lemma \ref{bpij}.
We will show in Subsection~\ref{cxytk} that the category $\Cal C$ has an involutive automorphism $F\colon\Cal C\to\Cal C$ defined by $F(D,+,-,0,\sigma,\circ,\tau\pi_2)=(D,+,-,0,\tau,\circ',\sigma\pi_2)$. Here the operation $\circ$ is necessarily the operation $\sigma\pi_1+\tau\pi_2$, and $\circ'$ is defined to be $\tau\pi_1+\sigma\pi_2$. The functor $F$ is the identity on morphisms.\end{example}

\begin{theorem}\label{new} Let $(D,+,-,0,\sigma,\circ,\cdot)$ be a left ditruss. The operation $\cdot$ is left distributive, that is, $a\cdot(b+c)=a\cdot b+a\cdot c$ for all $a,b,c\in D$, if and only if the operation $\circ$ is {\em left skew $\sigma$-distributive}, i.e.~$a\circ(b+c)=a\circ b-\sigma(a)+a\circ c$ for all $a,b,c\in D$.\end{theorem}

\begin{proof} One has $$a\cdot(b+c)=a\cdot b+a\cdot c$$ if and only if $$-\sigma(a)+a\circ(b+c)=-\sigma(a)+a\circ b-\sigma(a)+a\circ c,$$ that is, if and only if $$a\circ(b+c)=a\circ b-\sigma(a)+a\circ c.$$
\end{proof}

Let $(G,+)$ be a group, and $\circ$ and $\cdot$ be two binary operations on the set $G$ with $\cdot$ distributive over $+$. A necessary and sufficient condition for the existence of a mapping $\sigma\colon G\to G$ such that $\sigma\pi_1=\circ-\cdot$, i.e., such that  $(G,+,\sigma,\circ,\cdot)$ be a left ditruss, is that the difference $\circ-\cdot$ be a binary operation that does not depend on the second factor, i.e., such that $a\circ b-a\cdot b=a \circ c-a\cdot c$ for all $a,b,c\in G$. Equivalently, such that $a\circ b-a\cdot b=a \circ 0-a\cdot 0$ for all $a,b\in G$. Now left distributivity of $\cdot$ is equivalent to the fact that left multiplication $\lambda_a\colon G\to G$, $\lambda_a\colon b\mapsto a\cdot b$, is a group endomorphism of the group $(G,+)$ for every $a\in G$. In particular, we necessarily have that $a\cdot 0=0$ for every $a\in G$. Therefore, in the particular case of $\cdot$ left distributive, we have that  $(G,+,\sigma,\circ,\cdot)$ is a left ditruss for some unary operation $\sigma$ if and only if $a\circ b-a\cdot b=a \circ 0$ for every $a,b\in G$. Moreover, if this hold, then we must have necessarily that $\sigma(a)=a\circ 0$ for every $a\in G$. This proves that:

\begin{proposition} Let $(G,+)$ be a group,  and $\circ$ and $\cdot$ be two binary operations on $G$ with $\cdot$ left distributive. Then the following two conditions are equivalent:

\noindent{\rm (a)} $a\circ b-a\cdot b=a\circ c-a\cdot c$ for every $a,b,c\in G$, that is, the binary operation $\circ-\cdot$ depends only on the first factor.

\noindent{\rm (b)} $a\circ b-a\cdot b=a\circ 0$ for every $a,b,c\in G$.

Moreover, if the previous two equivalent conditions hold and the mapping $\sigma\colon G\to G$ is defined setting $\sigma(a)=a\circ 0$ for every $a\in G$, then $(G,+,\circ,\cdot,\sigma)$ is a left ditruss and $a\circ(b+c)=a\circ b-\sigma(a)+a\circ c$ for every $a,b,c\in G$.\end{proposition}

\begin{theorem}\label{2.3'}  Let $(D,+,-,0,\sigma, \circ,\cdot)$ be a left ditruss and suppose that the equivalent conditions of {\rm Theorem \ref{new}} hold, that is, assume that the operation $\cdot$ is left distributive. Then:
		
		\noindent{\rm (a)} For every $a\in D$, the mapping $\lambda^D_a\colon D\to D$, defined by
$\lambda^D_a(b)=a\cdot b$ for every $b\in D$, is a group endomorphism of the group $(D,+)$.

		\noindent{\rm (b)} $a\cdot 0=0$, $a\cdot(-b)=-(a\cdot b)$, $a\circ 0=\sigma(a)$,  and $a\circ(-b)=\sigma(a)-a\circ b+\sigma(a)$ for every $a,b\in D$.
		
		\noindent{\rm (c)} Assume that the mapping $\sigma$ is an idempotent group endomorphism of $(D,+)$. Then the operation $\circ$ is associative, that is, $(a\circ b)\circ c=a\circ(b\circ c)$ for all $a,b,c\in D$, if and only if the operation $\cdot$ is {\em left weakly $\sigma$-associative}, i.e.~$(\sigma(a)+a\cdot b)\cdot c=a\cdot(b\cdot c)$ for all $a,b,c\in D$. 
		
		Moreover, if the previous two equivalent conditions in {\rm (c)} hold, then $\sigma(a\cdot b)=a\cdot \sigma(b)$ for every $a,b\in D$.
		\end{theorem}
		
		\begin{proof} (a) is just a restatement of the hypothesis that the operation $\cdot$ is left distributive. 
		
		(b) follows from (a) and Identity (\ref{skew}). Moreover, $\sigma(a)=a\circ 0=a\circ(b+(-b))=a\circ b-\sigma(a)+a\circ(-b)$.
		
		(c) Assume that  $\sigma$  is an idempotent endomorphism of $(D,+)$. Then $(a\circ b)\circ c=a\circ(b\circ c)$  if and only if $\sigma(\sigma(a)+a\cdot b)+(\sigma(a)+a\cdot b)\cdot c=\sigma(a)+a\cdot(\sigma(b)+b\cdot c)$, that is, if and only if \begin{equation}\sigma(a)+\sigma(a\cdot b)+(\sigma(a)+a\cdot b)\cdot c=\sigma(a)+a\cdot \sigma(b)+a\cdot(b\cdot c).\label{vhil}\end{equation} For $c=0$, identity (\ref{vhil}) implies that $\sigma(a\cdot b)=a\cdot \sigma(b)$ for every $a,b\in D$.  Therefore identity (\ref{vhil})  is equivalent to $(\sigma(a)+a\cdot b)\cdot c=a\cdot(b\cdot c)$.	\end{proof}

		Notice that left weak $\sigma$-associativity, that is $(\sigma(a)+a\cdot b)\cdot c=a\cdot(b\cdot c)$, can be written equivalently as $(a\circ b)c=a(bc)$ (and in this form the name weak $\sigma$-associativity is justified), for all $a,b,c\in D$. Equivalently,  the two equivalent statements in Theorem \ref{2.3'}(c) say that the mapping $\lambda$ is a semigroup morphism of the semigroup $(D,\circ)$ into the semigroup $\End_\Gp(D,+)$ (with composition of endomorphisms).
		We include this fact as statement (a) in the following proposition.
		
		\begin{proposition}\label{2.3''}  Let $(D,+,-,0,\sigma, \circ,\cdot)$ be a left ditruss and suppose that $\sigma$ is an idempotent additive group endomorphism of $(D,+)$, $\circ$ is associative and $\cdot$ is left distributive. Then:
		
		\noindent{\rm (a)} The mapping $\lambda^D\colon (D,\circ)\to \End_\Gp(D,+)$, defined by $\lambda^D\colon a\mapsto \lambda^D_a$, is a semigroup morphism. That is, $(a\circ b)\cdot c=a\cdot(b\cdot c)$ for every $a,b,c\in D$.
		
		\noindent{\rm (b)} The group endomorphism $\lambda^D_0$ of the group $(D,+,-,0)$ is an idempotent group endomorphism.
		\end{proposition}
		
		\begin{proof} (a)  The position $a\mapsto \lambda^D_a$ defines a mapping of $D$ into $\End_\Gp(D,+)$ by Theorem~\ref{2.3'}(a). It is a semigroup morphism $(D,\circ)\to \End_\Gp(D,+)$ by Theorem~\ref{2.3'}(c).
		
		(b) We must show that $\lambda^D_0\circ\lambda^D_0=\lambda^D_0$, that is, that $0\cdot(0\cdot a)=0\cdot a$. But we have already remarked that $0\cdot a=0\circ a$ for all $a$, so that $0\cdot(0\cdot a)=0\cdot a$ is equivalent to $0\circ(0\circ a)=0\circ a$. Now $\circ$ is associative, hence $\lambda^D_0\circ\lambda^D_0=\lambda^D_0$.
		\end{proof}
		
		For the next two paragraphs suppose that the left ditruss $D$ satisfies the hypotheses of Proposition~\ref{2.3''}. Then the two idempotent group endomorphisms $\sigma$ and $\lambda^D_0$ are right multiplication by $0_D$ and left multiplication by $0_D$ respectively (both with respect to the multiplication $\circ$). These two idempotent endomorphisms of the group $(D,+)$ are not necessaily equal, as the following example show. Let $\sigma$ be a nontrivial idempotent endomorphism of a group $(D,+,-,0_D)$. Consider the left ditruss $(D,+,-,0,\sigma, \sigma\pi_1, \circ_{0})$, where $\circ_{0}$ is the operation defined by $a\circ_0b=0$ for every $a,b\in D$. Then $\lambda^D_0$ is the null endomorphism of $(D,+)$, which is different from $\sigma$.
		
		Nevertheless, we have already remarked that the two idempotent group endomorphisms $\sigma$ and $\lambda^D_0$ always commute, because $0_D\circ(a\circ 0_D)=(0_D\circ a)\circ 0_D$. 
		
		\subsection{Properties of left ditrusses with $\cdot=\tau\pi_2$}\label{cxytk} In Example~\ref{cvhli} we have considered the ditrusses in the category $\Cal C$ of all the left ditrusses $(D,+,-,0,\sigma,\circ,\tau\pi_2)$ for some mapping $\tau\colon D\to D$. There is an involutive category automorphism $F\colon\Cal C\to\Cal C$ defined by $$F(D,+,-,0,\sigma,\circ,\tau\pi_2)=(D,+,-,0,\tau,\circ',\sigma\pi_2).$$ Here the operations are $\circ=\sigma\pi_1+\tau\pi_2$, and $\circ'=\tau\pi_1+\sigma\pi_2$. Thus $(D,+,-,0,\sigma,\circ,\tau\pi_2)$ corresponds to the difference $\sigma\pi_1=(\sigma\pi_1+\tau\pi_2)-\tau\pi_2$ and
		$(D,+,-,0,\tau,\circ',\sigma\pi_2)$  corresponds to the difference $\tau\pi_1=(\tau\pi_1+\sigma\pi_2)-\sigma\pi_2$. Here $\sigma$ and $\tau$ are arbitrary mappings $T\to T$.
		
		\medskip
		
		In the following proposition we have collected some properties of the left ditrusses in the category $\Cal C$.
		
		\begin{proposition}\label{3.2} Let $(D,+,-,0,\sigma,\circ,\tau\pi_2)$ be a left ditruss where $\sigma$ and $\tau$ are arbitrary mappings $D\to D$. Then:
		
	\noindent{\rm (a)} The binary operation $\tau\pi_2$ is left distributive if and only if $\tau$ is an endomorphism of the group $(D,+)$.
	
		\noindent{\rm (b)} Assume $\sigma(0)=\tau(0)=0$. If the binary operation $\circ=\sigma\pi_1+\tau\pi_2$ is associative, then $\sigma$ and $\tau$ are idempotent endomappings of $D$ and $\sigma\tau=\tau\sigma$.
		
		\noindent{\rm (c)} Let $\sigma$ and $\tau$ be two idempotent endomorphisms of the group $(D,+)$ and suppose $\sigma\tau=\tau\sigma$. Then the binary operation $\circ=\sigma\pi_1+\tau\pi_2$ is associative.\end{proposition}
		
		\begin{proof} (a)  is trivial, because $a(\tau\pi_2)(b+c)=a(\tau\pi_2)b+a(\tau\pi_2)c$ if and only if $\tau(b+c)=\tau(b)+\tau(c)$.
		
		As far as (b) and (c) are concerned, notice that $\circ=\sigma\pi_1+\tau\pi_2$ is associative if and only if $$a(\sigma\pi_1+\tau\pi_2)(b(\sigma\pi_1+\tau\pi_2)c)=a(\sigma\pi_1+\tau\pi_2)(\sigma(b)+\tau(c))=\sigma(a)+\tau((\sigma(b)+\tau(c))$$ is equal to 
		 $$(a(\sigma\pi_1+\tau\pi_2)b)(\sigma\pi_1+\tau\pi_2)c=(\sigma(a)+\tau(b))(\sigma\pi_1+\tau\pi_2)c=\sigma(\sigma(a)+\tau(b))+\tau(c).$$ Therefore $\circ$ is associative if and only if \begin{equation}\sigma(a)+\tau((\sigma(b)+\tau(c))=\sigma(\sigma(a)+\tau(b))+\tau(c).\label{vo}\end{equation}
		 
		 Now in order to prove (b) suppose $\sigma(0)=\tau(0)=0$ and $\circ$ associative, so that Identity (\ref{vo}) holds for every $a,b,c\in D$. Replacing $(a,b,c)$ in  Identity (\ref{vo}) with $(a,0,0),(0,b,0),(0,0,c)$ respectively, we get that $\sigma(a)=\sigma(\sigma(a))$, $\tau(\sigma(b))=\sigma(\tau(b))$ and $\tau(\tau(c))=\tau(c)$, respectively. Therefore $\sigma$ and $\tau$ are idempotent and commute. This proves (b).
		 
		 To prove (c), suppose that $\sigma$ and $\tau$ are two idempotent endomorphisms of $(D,+)$ that commute. In order to prove that the operation $\circ=\sigma\pi_1+\tau\pi_2$ is associative, we must prove that Identity (\ref{vo}) holds  for every $a,b,c\in D$. But $\sigma(a)+\tau(\sigma(b)+\tau(c))=\sigma(a)+\tau(\sigma(b))+\tau(\tau(c))=\sigma(\sigma(a))+\sigma(\tau(b))+\tau(c)=\sigma(\sigma(a)+\tau(b))+\tau(c)$, as desired.\end{proof}
		
		Let us go back to left ditrusses $(T,+,-,0,\sigma, \circ,\cdot)$. In Proposition~\ref{2.3''}(a)  we saw that, under suitable hypothesis, the mapping $\lambda^T\colon (T,\circ)\to \End_\Gp(T,+)$ is a semigroup morphism. 	Now, for any two semigroups $T$ and $E$, the most natural examples of semigroup morphisms $T\to E$ are the {\em constant morphisms}, that is, the morphisms $\varphi\colon T\to E$ such that $\varphi(t)=\varphi(t')$ for every $t,t'\in T$. Clearly, for $T$ non-empty, constant morphisms $T\to E$ are in one-to-one correspondence with the idempotent elements of $E$. 
		
		%In the next corollary we give further descriptions of the objects in the category $\Cal C$ already described in Lemma~\ref{bpij} under the further condition that the endomapping $\sigma$ is an idempotent group endomorphism.
		
		\begin{corollary}\label{3.9} The following conditions are equivalent for an algebra $(D,+,-,0,\circ,\sigma)$ with $(D,+,-,0)$ a group, $\circ$ a binary operation on $D$ and $\sigma$ an idempotent group endomorphism of $(D,+,-,0)\!:$
		
\noindent{\rm (a)} $(D,+,-,0,\circ,\sigma)$ is a left skew truss for which the semigroup endomorphism $\lambda^D\colon (D,\circ)\to\End_{\Gp}(D,+)$ is a constant mapping.

\noindent{\rm (b)} There exists an idempotent group endomorphism $\tau$ of $(D,+,-,0)$ that commutes with $\sigma$ and such that $\circ=\sigma\pi_1+\tau\pi_2$.

\noindent{\rm (c)} There exist an idempotent group endomorphism $\tau$ of $(D,+,-,0)$ that commutes with $\sigma$ and a binary operation $\circ'$ on $D$ such that $(D,+,-,0,\circ',\tau)$ is a left skew truss for which the corresponding semigroup endomorphism $\lambda^D\colon (D,\circ')\to\End_{\Gp}(D,+)$ is the mapping constantly equal to $\sigma$.

\noindent{\rm (d)} There exists an idempotent group endomorphism $\tau$ of $(D,+,-,0)$ that commutes with $\sigma$ and such that $(D,+,-,0,\tau\pi_1+\sigma\pi_2,\tau)$ is a left skew truss.

Moreover, if these equivalent conditions {\rm (a)}-{\rm (d)} hold, then the operation $\circ$ is not only left skew $\sigma$-distributive, but also right skew $\tau$-distributive, and the operation $\circ'$ is not only left skew $\tau$-distributive, but also right $\sigma$-distributive.
\end{corollary}
		
		\begin{proof} (a)${}\Rightarrow{}$(b) Suppose that (a) holds, so that $(D,+,-,0,\circ,\sigma)$ is a left skew truss and there exists a group endomorphisms $\tau$ of $(D,+)$ such that $\lambda^D_a=\tau$ for every $a\in D$. Since $\lambda^D\colon (D,\circ)\to\End_{\Gp}(D,+)$ is a semigroup endomorphism, we have that $\tau$ is necessarily idempotent. As $(D,+,-,0,\circ,\sigma)$ is a left skew truss, the operation $\circ$ is associative and $\sigma\pi_1=\circ-\cdot$, so that $a\circ b=a(\sigma\pi_1)b+a\cdot b=\sigma(a)+\lambda^D_a(b)=\sigma(a)+\tau(b)$. Therefore $\circ=\sigma\pi_1+\tau\pi_2$. By Proposition~\ref{3.2}(b), the idempotent endomorphisms $\sigma$ and $\tau$ commute.
		
		(b)${}\Rightarrow{}$(c) Let $\tau$ be an idempotent group endomorphism of $(D,+,-,0)$ that commutes with $\sigma$ and such that $\circ=\sigma\pi_1+\tau\pi_2$. We will prove that the binary operation $\circ'=\tau\pi_1+\sigma\pi_2$ on $D$ is such that $(D,+,-,0,\circ',\tau)$ is a left skew truss for which the corresponding semigroup endomorphism $\lambda^D\colon (D,\circ')\to\End_{\Gp}(D,+)$ is the mapping constantly equal to $\sigma$. From Example~\ref{cvhli} we knwo that $(D,+,-,0,\tau,\circ',\sigma\pi_2)$ is a left ditruss. By  Proposition~\ref{3.2}(c), the operation $\circ'$ is associative. Also $a\circ'(b+c)=\tau(a)+\sigma(b+c)=\tau(a)+\sigma(b)+\sigma(c)=\tau(a)+\sigma(b)-\tau(a)+\tau(a)+\sigma(c)=a\circ'b-\tau(a)+a\circ'c$, so that $(D,+,-,0,\circ',\tau)$ is a left skew truss. Finally, for the corresponding semigroup morphism $\lambda$, we have that $\lambda_a(b)=-\tau(a)+a\circ'b=-\tau(a)+\tau(a)+\sigma(b)=\sigma(b)$.

			(c)${}\Rightarrow{}$(d) 	Assume that (c) holds. In order to prove (d) it suffices to show that $\circ'=\tau\pi_1+\sigma\pi_2$. But since $(D,+,-,0,\circ',\tau)$ is a left skew truss, we know that $a\circ'b=\tau(a)+a\cdot b=\tau(a)+\lambda_a(b)=\tau(a)+\sigma(b)=a(\tau\pi_1+\sigma\pi_2)b$, as desired.
			
			(d)${}\Rightarrow{}$(a) Apply the implications (a)${}\Rightarrow{}$(b)${}\Rightarrow{}$(c) to the left skew truss $(D,+,-,0,\tau\pi_1+\sigma\pi_2,\tau)$, with the roles of the two commuting idempotent group morphisms $\sigma$ and $\tau$ interchanged.
		\end{proof}

		\section{Interchange near rings, weak right trusses}\label{Inter}
		
		Recall that for any binary operation $\circ$ on a set $D$, it is possible to define its opposite $\circ^{\op}$ setting $a\circ^{\op}b=b\circ a$ for all $a,b\in D$.
		
		\begin{proposition} Let $(D,+,-,0,\sigma,\circ,\tau\pi_2)$ be a left ditruss with $\sigma$ and $\tau$ arbitrary mappings $D\to D$, so that $\circ=\sigma\pi_1+\tau\pi_2$. Assume $\circ'=\tau\pi_1+\sigma\pi_2$.  Then $\circ'=\circ^{\op}$ if and only if $\sigma$ and $\tau$ are {\em image-commuting}, that is, $\sigma(x)+\tau(y)=\tau(y)+\sigma(x)$ for every $x,y\in D$.\end{proposition}
		
		\begin{proof} $\circ'=\circ^{\op}$ if and only if $a\circ'b=a\circ^{\op}b$ for every $a,b\in D$, that is, if and only if $a(\tau\pi_1+\sigma\pi_2)b=b\circ a$. This is equivalent to $\tau(a)+\sigma(b)=\sigma(b)+\tau(a)$.\end{proof}
		
		Notice that, as far as the involutive category automorphism described in Corollary~\ref{3.9} is concerned (it transforms the operation $\circ=\sigma\pi_1+\tau\pi_2$ into the operation $\circ'=\tau\pi_1+\sigma\pi_2$), we necessarily have that $\circ$ must be left skew $\sigma$-distributive and right skew $\tau$-distributive.
		
		\begin{example} {\em Interchange near-rings, associative interchange near-rings.} An {\em interchange near-ring} $(G,+,-,0_G,\circ)$ is a group $(G,+,-,0_G)$ with a further binary operation $\circ$ for which the {\em interchange law} \begin{equation}(w+x)\circ(y+z)=(w\circ y)+(x\circ z)\label{interchange}\end{equation} holds, for all $w,x,y,z\in G$ \cite{Australian}. For an interesting excursus on the history of the interchange law, see the Introduction of \cite{Australian}. An  interchange near-ring $(G,+,-,0_G,\circ)$ is an {\em associative} interchange near-ring if the operation $\circ$ is associative. 
		
		According to one of the main results of \cite{Australian} (Theorem~3.5), the interchange near-ring structures on any group $(G,+,-,0_G)$ are all of the form $(G,+,-,0_G, \varepsilon\pi_1+\eta\pi_2)$ for a pair $(\varepsilon,\eta)$ of image-commuting group endomorphisms of 	$(G,+,-,0_G)$. The interchange near-ring $(G,+,-,0_G, \varepsilon\pi_1+\eta\pi_2)$ is associative if, moreover, $\varepsilon$ and $\eta$ are commuting idempotent group endomorphisms of 	$(G,+,-,0_G)$ (Proposition~\ref{3.2}((b) and (c)) and \cite[Theorem~4.3]{Australian}).  
		
		\medskip
		
		Let $(T,+,-,0_T)$ be a group and $\sigma,\tau$ be two image-commuting idempotent endomorphisms of the additive group $T$, and consider a left ditruss $(T,+,-,0,\sigma, \circ,\cdot)$ with $\circ$ associative and $\lambda^T\colon (T,\circ)\to \End_\Gp(T,+)$ the constant semigroup morphism constantly equal to $\tau$. This simply means that the multiplication $\cdot$ is the operation $\tau\pi_2$, which is always associative, left weakly $\sigma$-associative, and left distributive. As a consequence, one has that $a\circ b=\sigma(a)+\tau(b)$. In particular $0_T\circ a=\tau(a)$ and $a\circ 0_T=\sigma(a)$. From associativity we know that $0_T\circ(a\circ 0_T)=(0_T\circ a)\circ 0_T$, that is $\sigma(\tau(a))=\tau(\sigma(a))$, so that the idempotent endomorphisms $\sigma$ and $\tau$ commute. From \cite[Theorem~4.3]{Australian}, we get that $(T,+,-,0_T,\circ)$ is an associative interchange near-ring.
		
		Conversely, if $(T,+,-,0_T,\circ)$ is an associative interchange near-ring, there exist two  image-commuting idempotent commuting endomorphisms $\sigma$ and $\tau$ of the group $(T,+,-,0_T,\circ)$ such that $a\circ b=\sigma(a)+\tau(b)$ for every $a,b\in T$. It follows that:
		
		\begin{theorem}\label{bjip} The following two categories are isomorphic:
		
		\noindent{\rm (a)} The category of all left ditrusses $(T,+,-,0,\sigma, \circ,\cdot)$ with $\circ$ associative, $$\lambda^T\colon (T,\circ)\to \End_\Gp(T,+)$$ a constant semigroup morphism, and $\sigma$ and $\lambda^D_0$ image-commuting idempotent endomorphisms of the group $(T,+,-,0)$.
		
		\noindent{\rm (b)} The category of all associative interchange near-rings.
\end{theorem}		
		\end{example}
		
		Now if $(T,+,-,0_T,\circ)$ is an associative interchange near-ring and we replace the multiplication $\circ$ with its opposite operation $\circ^{\op}$, defined by $a\circ^{\op} b:=b\circ a$ for all $a,b\in T$, Identity (\ref{interchange}) becomes $(y+z)\circ^{\op} (w+x)=(y\circ^{\op} w)+(z\circ^{\op} x)$. It follows that if $(T,+,-,0_T,\circ)$ is an associative interchange near-ring, then $(T,+,-,0_T,\circ^{\op})$ is also an associative interchange near-ring. That is, the functor $(T,+,-,0_T,\circ)\mapsto(T,+,-,0_T,\circ^{\op})$ is an involutive categorical isomorphism of the category of all associative interchange near-rings into itself. 
		
		From Theorem~\ref{bjip} we get that the functor $(T,+,-,0,\sigma, \circ,\tau\pi_2)\mapsto (T,+,-,0,\tau, \circ,\sigma\pi_2)$  is an involutive categorical isomorphism of the category of all left ditrusses $(T,+,-,0,\sigma, \circ,\cdot)$ with $\circ$ associative, $\lambda^T\colon (T,\circ)\to \End_\Gp(T,+)$ a constant semigroup morphism constantly equal to $\tau$, and $\sigma,\tau$ image-commuting. 
		
				\bigskip
				
				Let $(T,+,-,0_T,\sigma, \circ,\cdot)$ be a left ditruss with $\sigma$ an idempotent additive group endomorphism of $(T,+)$, $\circ$ associative and $\cdot$ left distributive.  Since $\lambda^T_0$ is an idempotent group endomorphism of the additive group of $T$, we know that $\lambda^T_0$ corresponds to a semidirect-sum decomposition of the group $(T,+)$. That is, $$T_0:=\ker(\lambda^T_0)=\{\,a\in T\mid 0\cdot a=0\,\}=\{\,a\in T\mid 0\circ a=0\,\}$$ is a normal subgroup of $(T,+)$, $$T_c:=\lambda^T_0(T)=\{\,a\in T\mid 0\cdot a=a\,\}=\{\,a\in T\mid 0\circ a=a\,\}$$ is a subgroup of $(T,+)$, and $T$ is the semidirect sum $T=T_0\rtimes T_c$ as an additive group. We say that $T_0$ is the $0${\em -symmetric part }of the left ditruss $(T,+,-,0,\sigma,\circ,\cdot)$, and $T_c$ is its {\em constant part}. Both $T_0$ and $T_c$ are subditrusss of $(T,+,-,0,\sigma, \circ,\cdot)$, in the sense that they are additive subgroups of $(T,+)$ and are closed for the three operations $\sigma$, $\circ$ and $\cdot$.
						
						\bigskip
						
		A {\em left  weak truss} is an algebra $(W, +,-,0,\cdot,\sigma)$ in which

\noindent (1) $(W, +,-,0)$ is a (not-necessarily abelian) group;

\noindent (2) $\sigma$ is a unary operation;

\noindent (3) the binary operation $\cdot$ is  {\em left weakly $\sigma$-associative}, that is, $(\sigma(a)+a\cdot b)\cdot c=a\cdot(b\cdot c)$ for all $a,b,c\in W$; and 

\noindent (4) {\em left distributivity} holds, that is, $a(b+c)=ab+ac$ for every $a,b,c\in W$.

\begin{theorem}\label{main} The category of the left skew trusses $(T, +,-,0_T,\circ,\sigma)$ with $\sigma$ an idempotent group endomorphism of $(T, +,-,0_T)$ and the category of the left  weak trusses $(W, +,-,0,\cdot,\sigma)$  with $\sigma$ an idempotent group endomorphism of $(W, +,-,0_W)$
are canonically isomorphic. The canonical isomorphism associates to every left skew truss $(T, +,-,0_T,\circ,\sigma)$ with $\sigma$ an idempotent group endomorphism the left weak truss $(T, +,-,0_T,\cdot,\sigma)$, where $\cdot$ is the binary operation on $T$ defined setting $a\cdot b=-\sigma(a)+a\circ b$ for all $a,b\in T$, and associates to each left skew truss morphism $f\colon T\to T'$ the same mapping~$f$. \end{theorem}
		
	\begin{proof} Let $(T, +,\circ,\sigma)$ be a left skew truss with $\sigma$ an idempotent group endomorphism of $(T, +)$. Let $\cdot$ be the binary operation on $T$ defined setting $a\cdot b=-\sigma(a)+a\circ b=\lambda_a(b)$ for all $a,b\in T$, so that $(T, +,\sigma,\circ,\cdot)$ is a left ditruss in which $\circ$ is associative and $\cdot$ is left distributive. Then  $(T, +,\cdot,\sigma)$ is a left weak truss by Theorems~\ref{new} and \ref{2.3'}(c). Notice that a mapping $f\colon T\to T'$, $T'$ a left skew truss, respects $+$, $\circ$ and $\sigma$ if and only if it respects any two of the three operations $+$, $\circ$ and $\sigma$. Thus we have a canonical functor $(T, +,\circ,\sigma)\mapsto (T', +,\cdot,\sigma)$ of the category of the left skew trusses $(T, +,-,0_T,\circ,\sigma)$ with $\sigma$ an idempotent group endomorphism into the category of the left weak trusses $(W, +,-,0,\cdot,\sigma)$  with $\sigma$ an idempotent group endomorphism of the group $(W, +,-,0)$.
		
	Now let $(W, +,\cdot,\sigma)$ be a left weak truss with $\sigma$ an idempotent group endomorphism, and let $\circ$ be the binary operation on $W$ defined setting $a\circ b=\sigma(a)+a\cdot b$ for every $a,b\in W$. The operation $\circ$ is associative by Theorem~\ref{2.3'}(c) and left skew $\sigma$-distributive by Theorem~\ref{new}. Therefore $(W,+,\circ,\sigma)$ is a left skew truss, and we get a canonical functor $(W, +,\cdot,\sigma)\mapsto (W,+,\circ,\sigma)$ of  the category of the left weak trusses $(W, +,-,0_W,\circ,\sigma)$ with $\sigma$ an idempotent group endomorphism into the category of the left skew trusses $(T, +,-,0,\cdot,\sigma)$  with $\sigma$ an idempotent group endomorphism of the group $(T, +,-,0)$. This functor  is clearly an inverse of the functor described in the previous paragraph.
		\end{proof}
					
	\begin{examples}\label{3.10}{\rm As far as the ditrusses considered in Corollary~\ref{3.9} are concered, let us give three examples with three special cases.
		
		(1) The case $\sigma=\tau$. In this case our difference of operations $\sigma\pi_1=\circ-\cdot$ becomes $\sigma\pi_1=(\sigma\pi_1+\sigma\pi_2)-\sigma\pi_2$, with $\sigma$ an idempotent group endomorphism (so that $\circ$ is associative).  Then we have the left skew truss $(G,+,\circ,\sigma)$, where $\circ$ is defined by $a\circ b=\sigma(a+b)$. Its corresponding left weak truss is $(G,+,\sigma\pi_2,\sigma)$.
		
		(2) The case $\sigma=0$ and $\tau=\id_G$. Recall that, for $\sigma=0$, left skew trusses are exactly left near-rings (Example~(1) in Subsection~\ref{SS2.1}). In this case our difference of operations $\sigma\pi_1=\circ-\cdot$ is $0=\pi_2-\pi_2$, that is $\circ=\cdot=\pi_2$. In this case the left skew truss (=left near-ring) $(G,+,\pi_2,0)$ corresponds to the left weak truss $(G,+,\pi_2,\id_G)$.
		
		(3)  The case $\sigma=\id_G$ and $\tau=0$. Recall that, for $\sigma=\id_G$, left skew trusses are exactly left skew rings. In this case the difference of operations $\sigma\pi_1=\circ-\cdot$ is $\pi_1=\pi_1-\circ_0$, so that $\circ=\pi_1$ and $\cdot=\circ_0$. Thus the left skew truss (=left skew ring) $(G,+,\pi_1)$ corresponds to the left weak truss $(G,+,\circ_0)$.}\end{examples}
					
		\begin{examples}\label{mmm} 	Let  $(G,+)$ be a group. Let us go back to the triples of operations $(\sigma,\circ,\cdot)$ on $G$ with $\sigma$ an idempotent group endomorphism of $(G,+)$, $\circ$ binary and associative, $\cdot$ binary and left distributive, and $\sigma\pi_1=\circ-\cdot$ in the group of all binary operations on $G$. We already know that this implies that $\circ$ must be necessarily left skew $\sigma$-distributive and $\cdot$ must be left weakly $\sigma$-associative (Theorems~\ref{new} and \ref{2.3'}(c)).
Here are two examples.

\noindent(1) Let  $(G,+)$ be a group and $\sigma,\tau$ be two commuting idempotent group endomorphisms of $(G,+)$. For instance $\tau$ could be the identity $\id_G\colon G\to G$, or the zero mapping $G\to G$, or the idempotent endomorphism $\sigma$ itself, like in Examples~\ref{3.10}. Consider the difference $\sigma\pi_1=(\sigma\pi_1+\tau\pi_2)-\tau\pi_2$. The operation $\tau\pi_2$ is associative, left weakly $\sigma$-associative and left distributive. Thus $(G,+,\tau\pi_2)$ is a left near-ring and $(G,+,\tau\pi_2,\sigma)$ is a left weak truss.  The left skew truss corresponding to the left weak truss $(G,+,\tau\pi_2,\sigma)$ is $(G,+,\sigma\pi_1+\tau\pi_2,\sigma)$ (cf.~ Lemma~\ref{lemma}). The semigroup morphism $\lambda$ for this left skew truss is the constant morphism equal to $\tau$. Therefore, for the ditruss $(G,+,\sigma,\sigma\pi_1+\tau\pi_2,\tau\pi_2)$, we have that $G_0=\ker(\tau)$ and $G_c=\tau(G)$.

\noindent(2) Let  $(G,+)$ be a group and $\sigma,\tau$ be two commuting idempotent group endomorphisms of $(G,+)$. Consider the equality $\sigma\pi_1=(\tau\pi_2+\sigma\pi_1)-(-\sigma\pi_1+\tau\pi_2+\sigma\pi_1)$. Notice that $\tau\pi_2+\sigma\pi_1=(\tau\pi_1+\sigma\pi_2)^{\op}$, so that $\tau\pi_2+\sigma\pi_1$ is associative if and only if $\tau\pi_1+\sigma\pi_2$ is associative, which is true by Proposition~\ref{3.2}(c). Also $a(-\sigma\pi_1+\tau\pi_2+\sigma\pi_1)b$ is the conjugate $\tau(b)^{\sigma(a)}$ in the group $(G,+)$. Correspondingly, we have the left skew truss $(G,+,(\tau\pi_2+\sigma\pi_1),\sigma)$ and the left weak truss $(G,+,\cdot,\sigma)$ with $a\cdot b=-\sigma(a)+\tau(b)+\sigma(a)$ for every $a,b\in G$ (the conjugate of $\tau(b)$ via $\sigma(a)$). In this case, $\lambda_0=\tau$. The semigroup morphism $\lambda$ is constant if and only if $\lambda_a=\tau$ for every $a$, that is, if and only if $\tau(b)^{\sigma(a)}=\tau(b)$ for every $a,b\in G$, that is, if and only if $\sigma$ and $\tau$ are image-commuting, that is, in the case of the associative interchange near-rings considered in Theorem~\ref{bjip}. Notice that:

(a) $\sigma$ and $\id_G$ are image-commuting if and only if $\sigma(G)\subseteq Z(G)$, the center of $G$. Let us  prove that $\sigma(G)\subseteq Z(G)$ if and only if $(G,+)$ has a direct-sum decomposition $G=A\oplus B$, where $A$ is an abelian normal subgroup of $G$, $B$ is a normal subgroup of $G$, and $\sigma\colon G\to G$ is the group endomorphism defined by $\sigma(a+b)=a$ for every $a\in A$ and $b\in B$. In fact, assume that $\sigma(G)\subseteq Z(G)$. To the idempotent endomorphism $\sigma$, there corresponds a semidirect-product decomposition $G=\sigma(G)\ltimes\ker(\sigma)$, hence an action $\alpha\colon\sigma(G)\to\Aut(\ker(\sigma))$, $\alpha\colon a\in \sigma(G)\mapsto\,$conjugation by $a$. Since $\sigma(G)\subseteq Z(G)$, conjugation by an element $a\in\sigma(G)$ is the identity mapping of $\ker(\sigma)$. Therefore the elements of $\sigma(G)$ and those of $\ker(\sigma)$ commute, so that we have a direct-sum decomposition $G=\sigma(G)\oplus\ker(\sigma)$ of $G$. Moreover $\sigma(G)\subseteq Z(G)$ implies that $\sigma(G)$ is an abelian group. Set $A:=\sigma(A)$ and $B:=\ker(\sigma)$, so that $A$ and $B$ are normal subgroups of $G$, $A$ is abelian and, clearly, $\sigma\colon G\to G$ is the group endomorphism defined by $\sigma(a+b)=a$ for every $a\in A$ and $b\in B$. 

(b) $\sigma$ and $\sigma$ are image-commuting if and only if $\sigma(G)$ is abelian.

(c) $\sigma$ and the zero endomorphism of $G$  are always image-commuting.\end{examples}

We conclude with a characterization of $0$-symmetric left skew trusses.

%\begin{proposition} Let $(T,+,\circ,\sigma)$ be a left skew truss with $\sigma\colon T\to T$ be any mapping with $\sigma(0_T)=0_T$. The following conditions are equivalent:

%{\rm (a)} The left skew truss $(T,+,\circ,\sigma)$ is $0$-symmetric.

%{\rm (b)}  $\sigma$ is an idempotent endomorphism of the group $(T,+)$ and $\circ=\sigma\pi_1$.\end{proposition}

%\begin{proof} {\rm (a)}${}\Rightarrow{}${\rm (b)} If $(T,+,\circ,\sigma)$ is a $0$-symmetric left skew truss, then $\circ$ is associative and $0\circ c=0$ for every $c\in T$. Thus in the corresponding left ditruss $(T,+,\circ,\cdot,\sigma)$ we have that $a\circ(0\circ c)=(a\circ 0)\circ c$ for every $a,c\in T$, so that $a\circ 0=\sigma(a)\circ c$, that is $\sigma(a)=\sigma(a)\circ c$, so $\lambda_a(c)=0$  by Theorem \ref{2.3}. Therefore $\lambda$ is the semigroup morhism constantly equal to the zero endomorphism of $(G,+)$. It follows that $a\cdot c=0$, $a\circ c=\sigma(a)$, so $\circ=\sigma\pi_1$. Finally, $\sigma$ is an idempotent group endomorphism of $(G,+)$ by Lemma~\ref{1}.

%{\rm (a)}${}\Rightarrow{}${\rm (b)} is now trivial, because it is Example~\ref{mmm}(1) with $\tau$ the zero endomorphism. \end{proof}

\begin{proposition} Let $(T,+,\circ,\sigma)$ be a left skew truss with $\sigma\colon T\to T$ any mapping such that $\sigma(0_T)=0_T$. The following conditions are equivalent:

{\rm (a)} The left skew truss $(T,+,\circ,\sigma)$ is $0$-symmetric.

{\rm (b)}  $\sigma(a)\circ b=\sigma(a)$ for every $a,b\in T$.\end{proposition}

\begin{proof} {\rm (a)}${}\Rightarrow{}${\rm (b)} Let $(T,+,\circ,\sigma)$ be a $0$-symmetric left skew truss, and let $(T,+,\sigma,\circ,\cdot)$ be the corresponding left ditruss. Then $\circ$ is associative and $0\cdot b=0$ for every $b\in T$, that is $\lambda_0=0$ ($0$ is a two-sided zero for the magma $(T,\cdot)$). Since $\cdot$ is weakly $\sigma$-associative, we have that $(\sigma(a)+a\cdot 0)\cdot b=a\cdot(0\cdot b)$ for every $a,b\in T$. Therefore $\sigma(a)\cdot b=0$. Now $\sigma$ is an idempotent mapping by Theorem~\ref{2.3}(c), so the equality $\sigma(x)=x\circ y-x\cdot y$ implies that $\sigma(a)=\sigma(a)\circ b$.

{\rm (b)}${}\Rightarrow{}${\rm (a)} We know that $\sigma$ is an idempotent mapping by Theorem~\ref{2.3}(c). Thus the identity $\sigma(x)=x\circ y-x\cdot y$ implies that $\sigma(a)=\sigma(a)\circ b-\sigma(a)\cdot b$ for every $a,b\in T$. Hence (b) yields that $\sigma(a)\cdot b=0$. For $a=0$, we get that $0\cdot b=0$, so $T$  is $0$-symmetric.
\end{proof}


\begin{thebibliography}{99} 

\bibitem{Bre} T. Brzezi\'nski, {\em Trusses: between braces and rings}, Trans. Amer. Math. Soc. {\bf 372} (2019), no. 6, 4149--4176.

\bibitem{BS}  S. Burris and H. P. Sankappanavar, ``A Course in Universal Algebra'', Graduate Texts in Math. {\bf 78}, Springer-Verlag, New York - Berlin, 1981. Millennium Edition available at https://math.hawaii.edu/$\sim$ralph/Classes/619/univ-algebra.pdf

\bibitem{Australian} C. C. Edmunds, {\em Interchange rings}, J. Aust. Math. Soc. {\bf 101} (2016), 310--334.

\bibitem{Fac} A. Facchini, {\em Skew braces, near-rings, skew rings, dirings}, submitted for publication, 2025, available at arXiv:2507.22182.

\bibitem{semi} A. Facchini and D. Stanovsk\'{y}, {\em Semidirect products in Universal Algebra}, in ``Algebraic Structures and  Applications'', A. Laghribi and A. Leroy Eds., Contemporary Math. {\bf 826}, Amer. Math. Soc., Providence, 2025, pp. 103--124. 

\bibitem{GV} L. Guarnieri and L. Vendramin, {\em Skew braces and the Yang-Baxter equation,} Math. Comp.  {\bf 86} (2017), no. 307, 2519--2534. 
	
%\bibitem{26} 	P. J. Higgins,  {\em Groups with multiple operators}, Proc. London Math. Soc. (3) {\bf 6} (1956), 366--416.

%\bibitem{JMV2016} G. Janelidze, L. M\'arki and S. Veldsman, {\em Commutators for near-rings: Huq $\ne$ Smith}, Algebra Universalis {\bf 76} (2016), no. 2, 223--229.

\bibitem{LPOMR} S. R. L\'{o}pez-Permouth, I. Owusu-Mensah and A. Rafieipour, {\em  A monoid structure on the set of all binary operations over a fixed set},
Semigroup Forum {\bf 104} (2022), no. 3, 667--688.

%%\bibitem{loops} N. Martins-Ferreira and T. Van Der Linden, {\em A note on the ``Smith is Huq'' condition}, Appl. Cat. Struct. {\bf 20} (2012), no. 2, 175--187.

%\bibitem{OM}  I. Owusu-Mensah, ``Algebraic Structures on the Set of all Binary Operations over a Fixed Set'', Ph.D Thesis, Ohio University, 2020.

	\bibitem{RumpJAA2019} W. Rump, {\em Set-theoretic solutions to the Yang-Baxter equation, skew-braces, and related near-rings}, J. Algebra Appl. {\bf 18} (2019), no. 8, 1950145 (22 pages).
\end{thebibliography}
\end{document}